\documentclass[11pt, twosided]{article}
\usepackage{graphicx}
\usepackage{amsfonts,amsbsy,amssymb,amsmath}
\allowdisplaybreaks
\usepackage{cite}
\usepackage{graphics}
\def\bpsp{\begin{pspicture}}
\def\epsp{\end{pspicture}}

\newtheorem{theorem}{Theorem}[section]
\newtheorem{remark}[theorem]{Remark}
\newtheorem{example}[theorem]{Example}
\newtheorem{lemma}[theorem]{Lemma}
\newtheorem{corollary}[theorem]{Corollary}
\newtheorem{definition}[theorem]{Definition}
\newtheorem{proposition}[theorem]{Proposition}

\newtheorem{note}{Note}
\newtheorem{case}{Case}

\newtheorem{conjecture}{Conjecture}
\newtheorem{question}{Question}

\newcommand{\bea}{\begin{eqnarray}}
\newcommand{\eea}{\end{eqnarray}}
\newcommand{\beq}{\begin{eqnarray*}}
\newcommand{\eeq}{\end{eqnarray*}}

\def\m4{\mbox{\rm ~(mod $4$)}}

\def \bd{\begin{definition}}
\def \ed{\end{definition}}
\def \bqu{\begin{question}}
\def \equ{\end{question}}
\def \bcc{\begin{conjecture}}
\def \ecc{\end{conjecture}}
\def \bt{\begin{theorem}}
\def \et{\end{theorem}}
\def \bl{\begin{lemma}}
\def \el{\end{lemma}}
\def \bc{\begin{corollary}}
\def \ec{\end{corollary}}
\def \be{\begin{equation}}
\def \ee{\end{equation}}
\def \ben{\begin{enumerate}}
\def \een{\end{enumerate}}
\def \ba{\begin{array}}
\def \ea{\end{array}}
\def \bp{\begin{proposition}}
\def \ep{\end{proposition}}
\def \bx{\begin{example}}
\def \ex{\end{example}}
\def \br{\begin{remark}}
\def \er{\end{remark}}
\def \bdsc{\begin{description}}
\def \edsc{\end{description}}

\def \bn{\begin{case}}
\def \en{\end{case}}
\def \bnt{\begin{note}}
\def \ent{\end{note}}
\def\1{1\!\!1}

\def\mm2{\mbox{\rm ~(mod $2$)}}
\def\m4{\mbox{\rm ~(mod $4$)}}

\def\qed{\nolinebreak\hfill\rule{.2cm}{.2cm}\par\addvspace{.5cm}}

\def\m{\mu}

\def\1{\textbf{1}}
\def\0{\textbf{0}}

\begin{document}

\title{Graphs with $\{P_3,P_4,P_5\}$-factors in terms of size and spectral radius}
\author{Zahoor Iqbal Bhat $ ^{a} $, S. Pirzada$ ^{b} $\\
$^{a,b}${\em Department of Mathematics, University of Kashmir, Srinagar, India}\\
$^{a}$\texttt{zahoorbutt9642@gmail.com }; $ ^{b} $\texttt{pirzadasd@kashmiruniversity.ac.in}\\
}
\date{}

\pagestyle{myheadings} \markboth{Pirzada, Zahoor}{}
\maketitle
\vskip 5mm
\noindent{\footnotesize \bf Abstract.}
\vskip 3mm
\noindent Let $G$ be a connected graph of order $n$. A $\{P_3,P_4,P_5\}$-factor is a spanning subgraph $H$ of $G$ such that every component of $H$ is isomorphic to an element of
$\{P_3,P_4,P_5\}$. In this paper, we establish a sufficient condition on the size of the graph  $G$ with minimum degree $\delta$ to have a $\{P_3, P_4, P_5\}$-factor. Subsequently, we provide another sufficient
condition on the adjacency spectral radius, ensuring that a connected
graph $G$ with minimum degree $\delta$ contains a $\{P_3, P_4, P_5\}$-factor.

\noindent{\footnotesize Keywords: Factors , adjacency matrix, spectral radius, minimum degree}

\vskip 3mm
\noindent {\footnotesize AMS subject classification: 05C50, 05C12, 15A18.}

\section{Introduction}

Throughout, we consider finite simple graphs.
Let $G$ be a graph with vertex set $V(G)$ and edge set $E(G)$. Order and size of $G$ are respectively $|V|=n$ and $|E|=m$. The adjacency matrix $A(G)=[a_{ij}]$ of $G$ is an $n\times n$ matrix with $a_{ij}$ equal to $1$ or $0$ according as  $v_i$ and $v_j$ are adjacent or not. The eigenvalues of $G$ are the eigenvalues of its adjacency matrix $A(G)$. Let $\lambda_1, \lambda_2, \dots, \lambda_n$ be the eigenvalues of $A(G)$ in  non-increasing order. $\lambda_1$ is  the spectral radius of $G$ and is  written as $\rho(G)$. For a vertex $v \in V(G)$, the neighborhood of $v$ and the degree of $v$ in $G$
are denoted by $N_G(v)$ and $d_G(v)$, respectively. Let $i(G)$ denote the number of isolated vertices in $G$. For a subset $S \subseteq V(G)$, let $G[S]$ denote the subgraph of $G$ induced by $S$, and let $G-S$ denote the subgraph induced by $V(G)\setminus S$.
For two vertex-disjoint graphs $G_1$ and $G_2$, their union is denoted by $G_1 \cup G_2$. The disjoint union of $t$ copies of $G$, where $t$ is a positive integer, is denoted by $tG$.
The join $G_1 \vee G_2$ is the graph obtained by joining each vertex of $G_1$ to every vertex of $G_2$. We denote the path, cycle, star and complete graph of order $n$ by
$P_n$, $C_n$, $K_{1,n-1}$ and $K_n$, respectively. For a real number $c$ be a real number, $\lfloor c \rfloor$ denotes the greatest integer less than or equal to $c$.

\noindent Let $\mathcal{H}$ be a set of connected graphs.
A spanning subgraph $H$ of $G$ is called an $\mathcal{H}$-factor if every component of $H$
belongs to $\mathcal{H}$.
If $\mathcal{H}=\{P_3,P_4,P_5\}$, then an $\mathcal{H}$-factor is called a
$\{P_3,P_4,P_5\}$-factor. Write $P_{\ge k}=\{P_i \mid i\ge k\}$.
If $\mathcal{H}=P_{\ge k}$, then an $\mathcal{H}$-factor is called a $P_{\ge k}$-factor.
If $\mathcal{H}=\{K_2,C_i \mid i\ge 3\}$, then an $\mathcal{H}$-factor is called a
$\{K_2, C_i \mid i\ge 3\}$-factor.
If $\mathcal{H}=\{K_{1,j} \mid 1\le j\le k\}$, then an $\mathcal{H}$-factor is called a
$\{K_{1,j} \mid 1\le j\le k\}$-factor. Li and Miao \cite{1} and Zhou, Sun and Liu \cite{3} established several spectral radius conditions ensuring that a graph contains a $P_{\ge 2}$-factor. Subsequently, many researchers \cite{6,7} derived results concerning binding number, neighborhood conditions, and degree conditions for the existence of a $P_{\ge 3}$-factor in graphs. O \cite{9} investigated the relationship between the spectral radius and $\{K_{1,1}\}$-factors in graphs. Tutte \cite{10} characterized graphs containing a $\{K_2, C_n : n \ge 3\}$-factor. Klopp and Steffen \cite{11} examined the existence of a $\{K_{1,1}, K_{1,2}, C_m : m \ge 3\}$-factor in graphs. Amahashi and Kano \cite{12} established a necessary and sufficient condition for a graph to contain a $\{K_{1,j} : 1 \le j \le k\}$-factor, where $k \ge 2$ is an integer. Kano and Saito \cite{13} provided a sufficient condition for a graph to contain a $\{K_{1,j} : k \le j \le 2k\}$-factor, with $k \ge 2$. Furthermore, Kano, Lu and Yu \cite{14} explored the relationship between the number of isolated vertices and $\{K_{1,2}, K_{1,3}, K_5\}$-factors in graphs. Considerable attention has been devoted to identifying sufficient conditions for graphs to contain spanning subgraphs using various graph parameters such as spectral radius \cite{15,16,17}, independence number \cite{19,20}, degree conditions \cite{21}, isolated toughness \cite{23}, eigenvalues \cite{24,25}, and other parameters \cite{26,27}. Kano, Lu and Yu \cite{28} established a necessary and sufficient condition for a graph to contain a $\{K_{1,1}, K_{1,2}, \dots, K_{1,k}, T(2k+1)\}$-factor.\\

\noindent In 2010, Kano et al.~\cite{14} provided a sufficient condition for the existence of
$\{P_3, P_4, P_5\}$-factors in graphs.
\begin{theorem}{\em \cite{14}}
If a graph $G$ satisfies $i(G-S) \le \frac{2}{3}|S|$ for any subset  $S\subseteq V(G)$, Then $G$ contains $\{P_3, P_4, P_5\}$-factor.
\end{theorem}
\noindent Motivated by above mentioned work, we first establish a sufficient condition, in terms of size, that guarantees a connected graph $G$ with minimum degree $\delta$ to possess a
$\{P_3, P_4, P_5,\}$-factor. Furthermore, we present another sufficient
condition based on the adjacency spectral radius to ensure that a connected graph $G$ with minimum degree $\delta$ admits a $\{P_3, P_4, P_5,\}$-factor. Our main results are stated below.
\begin{theorem}
Let $G$ be a connected graph of order $n$ with $n \ge \frac{20\delta+53}{3}$, where $\delta$ is the minimum degree of $G$ and $\delta \not\equiv 0 \pmod{3}$. If
$E(G) > E(K_\delta \vee \left(K_{\,n-\left\lfloor \frac{5}{3}\delta \right\rfloor -1}
\cup \left(\left\lfloor \frac{2}{3}\delta \right\rfloor +1\right)K_1\right)$,
then $G$ has a $\Big\{ P_3, P_4, P_5 \Big\}$-factor.
\end{theorem}
\begin{theorem}
Let $G$ be a connected graph of order $n$ with $n \ge \max\left(\frac{6\delta^2+22\delta+31}{6}, \frac{20\delta+56}{3}\right)$, where $\delta$ is the minimum degree of $G$ and $\delta \not\equiv 0 \pmod{3}$. If $\rho(G) > \rho(K_\delta \vee \left(K_{\,n-\left\lfloor \frac{5}{3}\delta \right\rfloor -1}
\cup \left(\left\lfloor \frac{2}{3}\delta \right\rfloor +1\right)K_1\right)$,
then $G$ has a $\Big\{ P_3, P_4, P_5 \Big\}$-factor.
\end{theorem}
\section{The proof of Theorem 1.2}
\textit{Proof of Theorem 1.2.}
Suppose, to the contrary, that $G$ contains no $\{P_3,P_4,P_5\}$-factor.
Then, by Theorem 1.1, there exists a nonempty subset $S \subseteq V(G)$ such that $i(G-S) > \frac{2}{3}|S|$. Let $|S| = s$ and $i(G-S) = i$.
Since $i(G-S)$ is an integer, we have $i \ge \left\lfloor \frac{2}{3}s \right\rfloor + 1.$ It follows that $G$ is a spanning subgraph of $ G_1 = K_s \vee \left( K_{\,n-\left\lfloor \frac{5}{3}s \right\rfloor -1}
      \cup \left(\left\lfloor \frac{2}{3}s \right\rfloor +1\right)K_1 \right)$. So,\begin{equation}
E(G) \leq E(G_1).
\label{eq1}
\end{equation}
In addition, from $\delta(G)=\delta$ we have $\delta(G_1)=s \geq \delta$.\\

\noindent\textbf{Case 1}. $s=\delta$. Then $G_1= K_\delta \vee \left(K_{\,n-\left\lfloor \frac{5}{3}\delta \right\rfloor -1}
\cup \left(\left\lfloor \frac{2}{3}\delta \right\rfloor +1\right)K_1\right)$. From (2.1), we have \\
$E(G) \leq E(K_\delta \vee \left(K_{\,n-\left\lfloor \frac{5}{3}\delta \right\rfloor -1}
\cup \left(\left\lfloor \frac{2}{3}\delta \right\rfloor +1\right)K_1\right)$, which contradicts the conditions of the theorem.\\

\noindent\textbf{Case 2}. $s \ge \delta+1$. For $H=  K_x \vee \left( K_{\,n-\left\lfloor \frac{5}{3}x \right\rfloor -1}
      \cup \left(\left\lfloor \frac{2}{3}x \right\rfloor +1\right)K_1 \right)$ by simple calculations, we have\\
      $E(H)=\frac{1}{2}\Big\{ (n-\left\lfloor \frac{5}{3}x \right\rfloor)(n-\left\lfloor \frac{5}{3}x \right\rfloor-3)-x(x-2n+1)+2 \Big\}$. Now, we consider the following subcases.\\

\noindent\textbf{Subcase 2.1}. $s \equiv 0 \pmod{3}$ and $\delta \equiv 1 \pmod{3}$. Then $\left\lfloor \frac{5s}{3} \right\rfloor = \frac{5s}{3}$ and $\left\lfloor \frac{5\delta}{3} \right\rfloor = \frac{5\delta-2}{3}$. Also $n \ge \left\lfloor \frac{5s}{3} \right\rfloor +1$, which implies that $s \le \frac{3n-3}{5}$. Now, we have
\begin{align*}
& E\left(
K_\delta \vee
\left(
K_{\,n-\left\lfloor \frac{5\delta}{3} \right\rfloor -1}
\cup
\left(\left\lfloor \frac{2\delta}{3} \right\rfloor +1\right)K_1
\right)
\right) - E(G_1)  \\
&= \frac{1}{2}\Big\{ \left(n-\left\lfloor \frac{5}{3}\delta \right\rfloor\right)\left(n-\left\lfloor \frac{5}{3}\delta\right\rfloor-3\right)-\delta(\delta-2n+1)
+2 \Big\}\\
&\quad
-\frac{1}{2}\Big\{ \left(n-\left\lfloor \frac{5}{3}s \right\rfloor\right)\left(n-\left\lfloor \frac{5}{3}s \right\rfloor-3\right)-s(s-2n+1)+2 \Big\}\\
&= \frac{1}{2}\Big\{
\left(n-\frac{5\delta-2}{3}\right)
\left(n-\frac{5\delta-2}{3}-3\right)
-\left(n-\frac{5s}{3}\right)
\left(n-\frac{5s}{3}-3\right)
\Big\}  \\
&\quad
-\frac{1}{2}\Big\{
\delta(\delta-2n+1)-s(s-2n+1)
\Big\}  \\
&= \frac{1}{18}\Big\{
-16s^2+(12n-36)s+12n-12n\delta
+16\delta^2+16\delta-14
\Big\}.
\end{align*}

\noindent Let $f(s)=-16s^2+(12n-36)s+12n-12n\delta+16\delta^2+16\delta-14$. Since $\delta+1 \le s \le \dfrac{3n-3}{5}$, we have $f(\delta+1)-f\!\left(\frac{3n-3}{5}\right)
=
-\frac{4}{25}(3n-20\delta-53)(3n-5\delta-8)$. For $n \ge \dfrac{20\delta+53}{3}$, it follows that $f(s)$ attains its minimum
at $s=\delta+1$. Hence, we get
\begin{align*}
& E\left(
K_\delta \vee
\left(
K_{\,n-\left\lfloor \frac{5\delta}{3} \right\rfloor -1}
\cup
\left(\left\lfloor \frac{2\delta}{3} \right\rfloor +1\right)K_1
\right)
\right)
- E(G_1)  \\
&\ge \frac{1}{18}f(\delta+1)
= \frac{1}{18}(24n-52\delta-66)
\ge 0 .
\end{align*}

\noindent From (2.1), this implies that
\[
E(G) \le E(G_1) \le
E\left(
K_\delta \vee
\left(
K_{\,n-\left\lfloor \frac{5\delta}{3} \right\rfloor -1}
\cup
\left(\left\lfloor \frac{2\delta}{3} \right\rfloor +1\right)K_1
\right)
\right)
\]
which contradicts the conditions of the theorem.\\

\noindent\textbf{Subcase 2.2}. $s \equiv 0 \pmod{3}$ and $\delta \equiv 2 \pmod{3}$. Then$ \left\lfloor \frac{5s}{3} \right\rfloor = \frac{5s}{3}$ and $\left\lfloor \frac{5\delta}{3} \right\rfloor = \frac{5\delta-1}{3}$. Moreover, since $n \ge \left\lfloor \frac{5s}{3} \right\rfloor +1$, we obtain $s \le \frac{3n-3}{5}$. Now, we have
\begin{align*}
& E\!\left(
K_\delta \vee
\left(
K_{\,n-\left\lfloor \frac{5\delta}{3} \right\rfloor-1}
\cup
\left(\left\lfloor \frac{2\delta}{3} \right\rfloor+1\right)K_1
\right)
\right) - E(G_1) \\
&= \frac{1}{2}\Big\{
\left(n-\left\lfloor \tfrac{5\delta}{3} \right\rfloor\right)
\left(n-\left\lfloor \tfrac{5\delta}{3} \right\rfloor-3\right)
-\delta(\delta-2n+1)+2
\Big\} \\
&\quad
-\frac{1}{2}\Big\{
\left(n-\left\lfloor \tfrac{5s}{3} \right\rfloor\right)
\left(n-\left\lfloor \tfrac{5s}{3} \right\rfloor-3\right)
-s(s-2n+1)+2
\Big\} \\
&= \frac{1}{2}\Big\{
\left(n-\tfrac{5\delta-1}{3}\right)
\left(n-\tfrac{5\delta-1}{3}-3\right)
-\delta(\delta-2n+1)+2
\Big\} \\
&\quad
-\frac{1}{2}\Big\{
\left(n-\tfrac{5s}{3}\right)
\left(n-\tfrac{5s}{3}-3\right)
-s(s-2n+1)+2
\Big\} \\
&= \frac{1}{18}
\Big\{
-16s^2+(12n-36)s+6n-12n\delta
+16\delta^2+26\delta-8
\Big\}.
\end{align*}

\noindent Let $f(s)=-16s^2+(12n-36)s+6n-12n\delta+16\delta^2+26\delta-8$. Since $\delta+1 \le s \le \dfrac{3n-3}{5}$, we obtain $f(\delta+1)-f\!\left(\frac{3n-3}{5}\right)
=
-\frac{4}{25}(3n-20\delta-53)(3n-5\delta-8)$. For $n \ge \dfrac{20\delta+53}{3}$, it follows that $f(s)$ achieves its minimum value at $s=\delta+1$. Consequently,
\begin{align*}
& E\!\left(
K_\delta \vee
\left(
K_{\,n-\left\lfloor \frac{5\delta}{3} \right\rfloor-1}
\cup
\left(\left\lfloor \frac{2\delta}{3} \right\rfloor+1\right)K_1
\right)
\right) - E(G_1) \\
&\ge \frac{1}{18}f(\delta+1)
= \frac{1}{18}(18n-42\delta-60)
\ge 0 .
\end{align*}

\noindent Therefore, using (2.1), we obtain
\[
E(G) \le E(G_1) \le
E\!\left(
K_\delta \vee
\left(
K_{\,n-\left\lfloor \frac{5\delta}{3} \right\rfloor-1}
\cup
\left(\left\lfloor \frac{2\delta}{3} \right\rfloor+1\right)K_1
\right)
\right)
\]
which contradicts the hypothesis of the theorem.\\

\noindent \textbf{Subcase 2.3}. $s \equiv 1 \pmod{3}$ and $\delta \equiv 1 \pmod{3}$. Then $\left\lfloor \frac{5s}{3} \right\rfloor = \frac{5s-2}{3}$ and $\left\lfloor \frac{5\delta}{3} \right\rfloor = \frac{5\delta-2}{3}$. Since $n \ge \left\lfloor \frac{5s}{3} \right\rfloor +1$, it follows that $s \le \frac{3n-1}{5}$. Now, we have
\begin{align*}
& E\left(
K_\delta \vee
\left(
K_{\,n-\left\lfloor \frac{5\delta}{3} \right\rfloor-1}
\cup
\left(\left\lfloor \frac{2\delta}{3} \right\rfloor+1\right)K_1
\right)
\right) - E(G_1) \\
&= \frac{1}{2}\Big\{
\left(n-\left\lfloor \tfrac{5\delta}{3} \right\rfloor\right)
\left(n-\left\lfloor \tfrac{5\delta}{3} \right\rfloor-3\right)
-\delta(\delta-2n+1)+2
\Big\}  \\
&\quad
-\frac{1}{2}\Big\{
\left(n-\left\lfloor \tfrac{5s}{3} \right\rfloor\right)
\left(n-\left\lfloor \tfrac{5s}{3} \right\rfloor-3\right)
-s(s-2n+1)+2
\Big\} \\
&= \frac{1}{2}\Big\{
\left(n-\tfrac{5\delta-2}{3}\right)
\left(n-\tfrac{5\delta-2}{3}-3\right)
-\delta(\delta-2n+1)+2
\Big\}  \\
&\quad
-\frac{1}{2}\Big\{
\left(n-\tfrac{5s-2}{3}\right)
\left(n-\tfrac{5s-2}{3}-3\right)
-s(s-2n+1)+2
\Big\} \\
&= \frac{1}{18}
\Big\{
-16s^2-(12n-16)s+12n\delta+16\delta^2+16\delta
\Big\}.
\end{align*}

\noindent Define $f(s)=-16s^2+(12n-16)s-12n\delta+16\delta^2+16\delta$. Since $\delta+1 \le s \le \dfrac{3n-1}{5}$, we obtain $ f(\delta+1)-f\!\left(\frac{3n-1}{5}\right)
=
-\frac{4}{25}(3n-2\delta-36)(3n-5\delta-6)$. For $n \ge \dfrac{20\delta+53}{3}$, it follows that $f(s)$ reaches its minimum at $s=\delta+1$. Consequently,
\begin{align*}
& E\!\left(
K_\delta \vee
\left(
K_{\,n-\left\lfloor \frac{5\delta}{3} \right\rfloor-1}
\cup
\left(\left\lfloor \frac{2\delta}{3} \right\rfloor+1\right)K_1
\right)
\right) - E(G_1) \\
&\ge \frac{1}{18}f(\delta+1)
= \frac{1}{18}(12n-32\delta-32)
\ge 0 .
\end{align*}

\noindent Hence, from (2.1), we get
\[
E(G) \le E(G_1) \le
E\!\left(
K_\delta \vee
\left(
K_{\,n-\left\lfloor \frac{5\delta}{3} \right\rfloor-1}
\cup
\left(\left\lfloor \frac{2\delta}{3} \right\rfloor+1\right)K_1
\right)
\right)
\]
which contradicts the assumption of the theorem.\\

\noindent\textbf{Subcase 2.4}. $s \equiv 1 \pmod{3}$ and $\delta \equiv 2 \pmod{3}$. Then $\left\lfloor \frac{5s}{3} \right\rfloor = \frac{5s-2}{3}$ and $\left\lfloor \frac{5\delta}{3} \right\rfloor = \frac{5\delta-1}{3}$. Moreover, since $n \ge \left\lfloor \frac{5s}{3} \right\rfloor +1$, we obtain $s \le \frac{3n-1}{5}$. Now, we have
\begin{align*}
& E\left(
K_\delta \vee
\left(
K_{\,n-\left\lfloor \frac{5\delta}{3} \right\rfloor-1}
\cup
\left(\left\lfloor \frac{2\delta}{3} \right\rfloor+1\right)K_1
\right)
\right) - E(G_1) \\
&= \frac{1}{2}\Big\{
\left(n-\left\lfloor \tfrac{5\delta}{3} \right\rfloor\right)
\left(n-\left\lfloor \tfrac{5\delta}{3} \right\rfloor-3\right)
-\delta(\delta-2n+1)+2
\Big\}  \\
&\quad
-\frac{1}{2}\Big\{
\left(n-\left\lfloor \tfrac{5s}{3} \right\rfloor\right)
\left(n-\left\lfloor \tfrac{5s}{3} \right\rfloor-3\right)
-s(s-2n+1)+2
\Big\} \\
&= \frac{1}{2}\Big\{
\left(n-\tfrac{5\delta-1}{3}\right)
\left(n-\tfrac{5\delta-1}{3}-3\right)
-\delta(\delta-2n+1)+2
\Big\} \\
&\quad
-\frac{1}{2}\Big\{
\left(n-\tfrac{5s-2}{3}\right)
\left(n-\tfrac{5s-2}{3}-3\right)
-s(s-2n+1)+2
\Big\} \\
&= \frac{1}{18}
\Big\{
-16s^2+(12n-16)s-6n-12n\delta
+16\delta^2+26\delta+6
\Big\}.
\end{align*}

\noindent Let $f(s)=-16s^2+(12n-16)s-6n-12n\delta+16\delta^2+26\delta+6$. Since $\delta+1 \le s \le \dfrac{3n-1}{5}$, we obtain $f(\delta+1)-f\!\left(\frac{3n-1}{5}\right)
=
-\frac{4}{25}(3n-20\delta-36)(3n-5\delta-6)$. For $n \ge \dfrac{20\delta+53}{3}$, it follows that $f(s)$ attains its
minimum at $s=\delta+1$. Therefore, we obtain
\begin{align*}
& E\left(
K_\delta \vee
\left(
K_{\,n-\left\lfloor \frac{5\delta}{3} \right\rfloor-1}
\cup
\left(\left\lfloor \frac{2\delta}{3} \right\rfloor+1\right)K_1
\right)
\right) - E(G_1) \\
&\ge \frac{1}{18}f(\delta+1)
= \frac{1}{18}(6n-22\delta-26)
\ge 0.
\end{align*}

\noindent Consequently, using (2.1), we have
\[
E(G) \le E(G_1) \le
E\left(
K_\delta \vee
\left(
K_{\,n-\left\lfloor \frac{5\delta}{3} \right\rfloor-1}
\cup
\left(\left\lfloor \frac{2\delta}{3} \right\rfloor+1\right)K_1
\right)
\right)
\]
which contradicts the assumption of the theorem.\\

\noindent\textbf{Subcase 2.5}. $s \equiv 2 \pmod{3}$ and $\delta \equiv 1 \pmod{3}$. Then
$\left\lfloor \frac{5s}{3} \right\rfloor = \frac{5s-1}{3}$ and $\left\lfloor \frac{5\delta}{3} \right\rfloor = \frac{5\delta-2}{3}$. Furthermore, since $n \ge \left\lfloor \frac{5s}{3} \right\rfloor +1$, it follows that
$s \le \frac{3n-2}{5}$. Now, we have
\begin{align*}
& E\left(
K_\delta \vee
\left(
K_{\,n-\left\lfloor \frac{5\delta}{3} \right\rfloor-1}
\cup
\left(\left\lfloor \frac{2\delta}{3} \right\rfloor+1\right)K_1
\right)
\right) - E(G_1) \\
&= \frac{1}{2}\Big\{
\left(n-\left\lfloor \tfrac{5\delta}{3} \right\rfloor\right)
\left(n-\left\lfloor \tfrac{5\delta}{3} \right\rfloor-3\right)
-\delta(\delta-2n+1)+2
\Big\}  \\
&\quad
-\frac{1}{2}\Big\{
\left(n-\left\lfloor \tfrac{5s}{3} \right\rfloor\right)
\left(n-\left\lfloor \tfrac{5s}{3} \right\rfloor-3\right)
-s(s-2n+1)+2
\Big\} \\
&= \frac{1}{2}\Big\{
\left(n-\tfrac{5\delta-2}{3}\right)
\left(n-\tfrac{5\delta-2}{3}-3\right)
-\delta(\delta-2n+1)+2
\Big\}  \\
&\quad
-\frac{1}{2}\Big\{
\left(n-\tfrac{5s-1}{3}\right)
\left(n-\tfrac{5s-1}{3}-3\right)
-s(s-2n+1)+2
\Big\} \\
&= \frac{1}{18}
\Big\{
-16s^2+(12n-26)s+6n-12n\delta
+16\delta^2+16\delta-6
\Big\}.
\end{align*}

\noindent Let $f(s)=-16s^2+(12n-26)s+6n-12n\delta+16\delta^2+16\delta-6 $. Since $\delta+1 \le s \le \dfrac{3n-2}{5}$, we obtain $f(\delta+1)-f\!\left(\frac{3n-2}{5}\right)
=
-\frac{2}{25}(6n-40\delta-89)(3n-5\delta-7)$. For $n \ge \dfrac{20\delta+53}{3}$, it follows that $f(s)$ attains its
minimum at $s=\delta+1$. Hence, we have
\begin{align*}
& E\!\left(
K_\delta \vee
\left(
K_{\,n-\left\lfloor \frac{5\delta}{3} \right\rfloor-1}
\cup
\left(\left\lfloor \frac{2\delta}{3} \right\rfloor+1\right)K_1
\right)
\right) - E(G_1) \\
&\ge \frac{1}{18}f(\delta+1)
= \frac{1}{18}(18n-42\delta-48)
\ge 0 .
\end{align*}

\noindent Consequently, using (2.1)
\[
E(G) \le E(G_1) \le
E\!\left(
K_\delta \vee
\left(
K_{\,n-\left\lfloor \frac{5\delta}{3} \right\rfloor-1}
\cup
\left(\left\lfloor \frac{2\delta}{3} \right\rfloor+1\right)K_1
\right)
\right)
\]
which contradicts the assumption of the theorem.\\

\noindent \textbf{Subcase 2.6}. $s \equiv 2 \pmod{3}$ and $\delta \equiv 2 \pmod{3}$. Then
$\left\lfloor \frac{5s}{3} \right\rfloor = \frac{5s-1}{3}$ and $\left\lfloor \frac{5\delta}{3} \right\rfloor = \frac{5\delta-1}{3}$. Furthermore, since $n \ge \left\lfloor \frac{5s}{3} \right\rfloor +1$, it follows that
$s \le \frac{3n-2}{5}$. Now, we have
\begin{align*}
&E\left(
K_\delta \vee
\left(
K_{\,n-\left\lfloor \frac{5\delta}{3} \right\rfloor-1}
\cup
\left(\left\lfloor \frac{2\delta}{3} \right\rfloor+1\right)K_1
\right)
\right) - E(G_1) \\
&= \frac{1}{2}\Big\{
\left(n-\left\lfloor \tfrac{5\delta}{3} \right\rfloor\right)
\left(n-\left\lfloor \tfrac{5\delta}{3} \right\rfloor-3\right)
-\delta(\delta-2n+1)+2
\Big\}  \\
&\quad
-\frac{1}{2}\Big\{
\left(n-\left\lfloor \tfrac{5s}{3} \right\rfloor\right)
\left(n-\left\lfloor \tfrac{5s}{3} \right\rfloor-3\right)
-s(s-2n+1)+2
\Big\} \\
&= \frac{1}{2}\Big\{
\left(n-\tfrac{5\delta-1}{3}\right)
\left(n-\tfrac{5\delta-1}{3}-3\right)
-\delta(\delta-2n+1)+2
\Big\}  \\
&\quad
-\frac{1}{2}\Big\{
\left(n-\tfrac{5s-1}{3}\right)
\left(n-\tfrac{5s-1}{3}-3\right)
-s(s-2n+1)+2
\Big\} \\
&= \frac{1}{18}
\Big\{
-16s^2+(12n-26)s-12n\delta
+16\delta^2+26\delta
\Big\}.
\end{align*}

\noindent Let $f(s)=-16s^2+(12n-26)s-12n\delta+16\delta^2+26\delta$. Since $\delta+1 \le s \le \dfrac{3n-2}{5}$, we obtain $f(\delta+1)-f\!\left(\frac{3n-2}{5}\right)
=
-\frac{2}{25}(6n-40\delta-89)(3n-5\delta-7)$.
For $n \ge \dfrac{20\delta+53}{3}$, it follows that $f(s)$ attains its
minimum at $s=\delta+1$. Hence, we have
\begin{align*}
& E\left(
K_\delta \vee
\left(
K_{\,n-\left\lfloor \frac{5\delta}{3} \right\rfloor-1}
\cup
\left(\left\lfloor \frac{2\delta}{3} \right\rfloor+1\right)K_1
\right)
\right) - E(G_1) \\
&\ge \frac{1}{18}f(\delta+1)
= \frac{1}{18}(12n-32\delta-42)
\ge 0.
\end{align*}

\noindent Thus, using (2.1)
\[
E(G) \le E(G_1) \le
E\left(
K_\delta \vee
\left(
K_{\,n-\left\lfloor \frac{5\delta}{3} \right\rfloor-1}
\cup
\left(\left\lfloor \frac{2\delta}{3} \right\rfloor+1\right)K_1
\right)
\right)
\]
which contradicts the assumption of the theorem.\qed
\noindent \textbf{Remark 2.1} The requirement $E(G) > E\left(K_\delta \vee \left( K_{\,n-\left\lfloor \frac{5}{3}\delta \right\rfloor -1}
\cup \left(\left\lfloor \frac{2}{3}\delta \right\rfloor +1\right)K_1 \right)\right)$ in Theorem 1.2 cannot be substituted by $E(G) \geq E\left(K_\delta \vee \left( K_{\,n-\left\lfloor \frac{5}{3}\delta \right\rfloor -1}
\cup \left(\left\lfloor \frac{2}{3}\delta \right\rfloor +1\right)K_1 \right)\right)$. Indeed, let us select $S = V(K_\delta)$ in the graph $K_\delta \vee \left( K_{\,n-\left\lfloor \frac{5}{3}\delta \right\rfloor -1}
\cup \left(\left\lfloor \frac{2}{3}\delta \right\rfloor +1\right)K_1 \right)$. Then $|S| = \delta$, and it follows that,
$i(G - S) = \left\lfloor \frac{2}{3}\delta \right\rfloor + 1 \geq \frac{2\delta}{3} + \frac{1}{3} = \frac{2|S|}{3} + \frac{1}{3} > \frac{2}{3}|S|$. Consequently, by Theorem 1.1, the graph does not admit any $\{P_3, P_4, P_5\}$-factor. \qed

\section{Proof of Theorem 1.3}
\noindent We first provide some lemmas, which will be used in proof of this theorem.
\begin{lemma} {\em \cite{30}}
Let $G$ be a graph with $n$ vertices. Then $\rho(G) \le \sqrt{2e(G)-n+1}$.

\end{lemma}
\begin{lemma}{\em \cite{31}}
Let $G$ be a connected graph and let $H$ be a subgraph of $G$. Then $\rho(G) \ge \rho(H)$,where equality holds if and only if $G = H$.
\end{lemma}
\textit{Proof of Theorem 1.3.}
To the contrary assume, that $G$ has no $\{P_3,P_4,P_5\}$-factor. By Theorem 1.1, there exists a vertex subset $S\subseteq V(G)$ such $i(G-S) > \frac{2}{3}|S|$. Let $|S|=s$ and denote $i(G-S)=i$. Clearly $i(G-S)$ is a nonnegative integer. Hence, $i \ge \left\lfloor \frac{2}{3}s \right\rfloor +1$. It follows that $G$ is a spanning subgraph of $G_1 = K_s \vee \left(K_{\,n-\lfloor \frac{5}{3}s \rfloor -1}
\cup \left(\left\lfloor \frac{2}{3}s \right\rfloor +1\right)K_1\right)$.
Since $\delta(G)=\delta$, we have $\delta(G_1)=s \ge \delta$.
Applying Lemma 3.2 yields
\begin{align}
\rho(G) \le \rho(G_1)
\tag{3.1}
\end{align}
where equality holds if and only if $G=G_1$.
Next, we consider two cases depending on the value of $s$.\\

\noindent\textbf{Case 1.} $s=\delta$. In this situation,
\[
G_1 = K_\delta \vee \left(K_{\,n-\lfloor \frac{5}{3}\delta \rfloor -1}
\cup \left(\left\lfloor \frac{2}{3}\delta \right\rfloor +1\right)K_1\right).
\]

\noindent Combining this with (3.1), we obtain
\[
\rho(G) \le
\rho\!\left(
K_\delta \vee
\left(
K_{\,n-\lfloor \frac{5}{3}\delta \rfloor -1}
\cup
\left(\left\lfloor \frac{2}{3}\delta \right\rfloor +1\right)K_1
\right)
\right).
\]
This contradicts the assumption of the theorem.\\

\noindent\textbf{Case 2.} $s \geq \delta + 1$.
Observe that $K_{\delta} \vee \left( K_{\,n-\left\lfloor \frac{5\delta}{3}\right\rfloor -1}
\cup \left(\left\lfloor \tfrac{2\delta}{3}\right\rfloor +1\right)K_1 \right)$
contains $K_{\,n-\left\lfloor \frac{2\delta}{3}\right\rfloor -1}$ as a proper subgraph.
By Lemma 3.2 together with the hypothesis of the theorem, we obtain
\begin{align}
\rho(G)
&> \rho\!\left(
K_{\delta} \vee
\left(
K_{\,n-\left\lfloor \frac{5\delta}{3}\right\rfloor -1}
\cup
\left(\left\lfloor \tfrac{2\delta}{3}\right\rfloor +1\right)K_1
\right)
\right) \notag\\
&> \rho\!\left(K_{\,n-\left\lfloor \frac{2\delta}{3}\right\rfloor -1}\right) \notag\\
&= n-\left\lfloor \frac{2\delta}{3}\right\rfloor-2.
\tag{3.2}
\end{align}

\noindent Now, let $G_1 = K_s \vee \left( K_{\,n-\left\lfloor \frac{5s}{3}\right\rfloor -1}
\cup (\,\left\lfloor \tfrac{2s}{3}\right\rfloor +1\,)K_1 \right)$.
Applying Lemma 3.1 yields
\begin{align*}
\rho(G_1)
&\leq \sqrt{\,2e(G_1)-n+1\,} \\
&= \sqrt{
\left(n-\left\lfloor \tfrac{2s}{3}\right\rfloor -1\right)
\left(n-\left\lfloor \tfrac{2s}{3}\right\rfloor -2\right)
+2s\left(\left\lfloor \tfrac{2s}{3}\right\rfloor +1\right)
-n+1 }. \tag{3.3}
\end{align*}
\noindent We have the following subcases.\\
\noindent\textbf{Subcase 2.1.} $s \equiv 0 \pmod{3}$.
In this case, $\left\lfloor \frac{2s}{3} \right\rfloor=\frac{2s}{3}$.
Combining this with (3.3), we obtain
\begin{align}
\rho(G_1)
&\le \sqrt{\left(n-\left\lfloor \frac{2s}{3} \right\rfloor-1\right)
           \left(n-\left\lfloor\frac{2s}{3}\right\rfloor-2\right)
           +2s\left(\left\lfloor\frac{2s}{3}\right\rfloor+1\right)-n+1} \notag\\
&= \sqrt{\left(n-\frac{2s}{3}-1\right)
         \left(n-\frac{2s}{3}-2\right)
         +2s\left(\frac{2s}{3}+1\right)-n+1} \notag\\
&= \frac{1}{3}\sqrt{9n^2-12ns-36n+16s^2+36s+27}.
\tag{3.4}
\end{align}

\noindent Let $g(s)=16s^2-(12n-36)s+9n^2-36n+27$. Since $n \ge \left\lfloor \frac{5s}{3} \right\rfloor+1=\frac{5s}{3}+1$,
it follows that $\delta+1 \le s \le \frac{3n-3}{5}$.
A straightforward calculation gives
$ g(\delta+1)-g\!\left(\frac{3n-3}{5}\right)
=\frac{4}{25}(3n-20\delta-53)(3n-5\delta-8)$. Hence, when $n \ge \max\!\left\{\frac{6{\delta}^2+22\delta+31}{6},\,\frac{20\delta+56}{3}\right\}$,
we have $g(\delta+1) \ge g\!\left(\frac{3n-3}{5}\right)$. Therefore $g(s)$ attains its maximum at $\delta+1$, and by (3.1), (3.4) and  $n \ge \max\!\left\{\frac{6{\delta}^2+22\delta+31}{6},\,\frac{20\delta+56}{3}\right\}$, we obtain
\[
\begin{aligned}
\rho(G)
&\le \rho(G_1)
\le \frac{1}{3}\sqrt{g(\delta+1)} \\
&= \frac{1}{3}\sqrt{9n^2-12n\delta-48n+16{\delta^2+68\delta+79}} \\
&= \frac{1}{3}\sqrt{
9\left(n-\left\lfloor\frac{2}{3}\delta\right\rfloor-2\right)^2
-6n\left(2\delta -3\left\lfloor\frac{2}{3}\delta\right\rfloor+2\right)
-9\left\lfloor\frac{2}{3}\delta\right\rfloor\left(\left\lfloor\frac{2}{3}\delta\right\rfloor+4\right)
+16\delta^2+68\delta+43
} \\[4pt]
&\le \frac{1}{3}\sqrt{
9\left(n-\left\lfloor\frac{2}{3}\delta\right\rfloor-2\right)^2 -6n\left(2\delta -3\frac{(2\delta-1)}{3}+2\right)
-9\left(\frac{2\delta-2}{3}\right)\left(\frac{2\delta-2}{3}+4\right)
+16\delta^2+68\delta+43
}\\[4pt]
&= \frac{1}{3}\sqrt{
9\left(n-\left\lfloor\frac{2}{3}\delta\right\rfloor-2\right)^2
-18n+12\delta^2+52\delta+63
} \\[4pt]
&< n-\left\lfloor\frac{2}{3}\delta\right\rfloor-2 .
\end{aligned}
\]
This contradicts (3.2). \\

\noindent\textbf{Subcase 2.2.} $s \equiv 1 \pmod{3}$. In this case, $\left\lfloor \frac{2s}{3} \right\rfloor=\frac{2s-2}{3}$. Substituting this relation in (3.3), we obtain
\begin{align}
\rho(G_1)
&\le \sqrt{\left(n-\frac{2s}{3}-\frac{1}{3}\right)
\left(n-\frac{2s}{3}-\frac{4}{3}\right)
+2s\left(\frac{2s}{3}+\frac{1}{3}\right)-n+1} \notag\\
&= \frac{1}{3}\sqrt{9n^2-12ns-24n+16s^2+16s+13}.
\tag{3.5}
\end{align}

\noindent Define $g(s)=16s^2-(12n-16)s+9n^2-24n+13$. Because $n \ge \left\lfloor \frac{5\delta}{3} \right\rfloor +1
= \frac{5\delta-2}{3}+1$, and since $\delta \not\equiv 0 \pmod{3}$ while
$s \equiv 1 \pmod{3}$, it follows that $\delta+2 \le s \le \frac{3n-1}{5}$. A direct computation shows that $g(\delta+2)-g\!\left(\frac{3n-1}{5}\right)
=\frac{4}{25}(3n-20\delta-56)(3n-5\delta-11)$. Hence, whenever
 $n \ge \max\!\left\{\frac{6{\delta}^2+22\delta+31}{6},\,\frac{20\delta+56}{3}\right\}$, the inequality
$ g(\delta+2) \ge g\left(\frac{3n-1}{5}\right)$
holds. Consequently, the maximum of $g(s)$ is achieved at $s=\delta+2$.
Combining (3.1) with (3.5) and using   $n \ge \max\!\left\{\frac{6{\delta}^2+22\delta+31}{6},\,\frac{20\delta+56}{3}\right\}$, yields
\[
\begin{aligned}
\rho(G)
&\le \rho(G_1)
\le \frac{1}{3}\sqrt{g(\delta+2)} \\
&= \frac{1}{3}\sqrt{9n^2-12n\delta-48n+16{\delta^2+80\delta+109}} \\
&= \frac{1}{3}\sqrt{
9\left(n-\left\lfloor\frac{2}{3}\delta\right\rfloor-2\right)^2
-6n\left(2\delta -3\left\lfloor\frac{2}{3}\delta\right\rfloor+2\right)
-9\left\lfloor\frac{2}{3}\delta\right\rfloor\left(\left\lfloor\frac{2}{3}\delta\right\rfloor+4\right)
+16\delta^2+80\delta+73
} \\[4pt]
&\le \frac{1}{3}\sqrt{
9\left(n-\left\lfloor\frac{2}{3}\delta\right\rfloor-2\right)^2 -6n\left(2\delta -3\frac{(2\delta-1)}{3}+2\right)
-9\left(\frac{2\delta-2}{3}\right)\left(\frac{2\delta-2}{3}+4\right)
+16\delta^2+80\delta+73
}\\[4pt]
&= \frac{1}{3}\sqrt{
9\left(n-\left\lfloor\frac{2}{3}\delta\right\rfloor-2\right)^2
-18n+12\delta^2+64\delta+93
} \\[4pt]
&< n-\left\lfloor\frac{2}{3}\delta\right\rfloor-2 .
\end{aligned}
\]
This contradicts inequality (3.2).\\

\noindent\textbf{Subcase 2.3.} $s \equiv 2 \pmod{3}$. In this case, $\left\lfloor \frac{2}{3}s \right\rfloor = \frac{2s-1}{3}$
Combining these identities with (3.3), we derive
\begin{align}
\rho(G_1)
&\le
\sqrt{\left(n-\left\lfloor \frac{2}{3}s \right\rfloor-1\right)\left(n-\left\lfloor \frac{2}{3}s \right\rfloor-2\right)+2s\left(\left\lfloor \frac{2}{3}s \right\rfloor+1\right)-n+1} \notag\\
&=
\sqrt{\left(n-\frac{2}{3}s-\frac{2}{3}\right)\left(n-\frac{2}{3}s-\frac{5}{3}\right)+2s\left(\frac{2}{3}s+\frac{2}{3}\right)-n+1} \notag\\
&=
\frac{1}{3}\sqrt{9n^2-12ns-30n+16s^2+26s+19}.
\tag{3.6}
\end{align}

\noindent Let $g(s)=16s^2-(12n-26)s+9n^2-30n+19$. Since $n \ge \left\lfloor \frac{5}{3}s \right\rfloor +1
= \frac{5s-1}{3}+1$, it follows that $\delta+1 \le s \le \frac{3n-2}{5}$. Hence, by a straightforward calculation, we obtain\\
$g(\delta+1)-g(\frac{3n-2}{5})=\frac{2}{25}(3n-5\delta-7)(6n-40\delta-89)$, for  $n \ge \max\!\left\{\frac{6{\delta}^2+22\delta+31}{6},\,\frac{20\delta+56}{3}\right\}$. We have $ g(\delta+1)\geq g(\frac{3n-2}{5})$. So, $g(s)$ has maximum at $\delta+1$. Thus from (3.1) and (3.6) and $n \ge \max\left\{\frac{6{\delta}^2+22\delta+31}{6},\,\frac{20\delta+56}{3}\right\}$, we have
\[
\begin{aligned}
\rho(G)
&\le \rho(G_1)
\le \frac{1}{3}\sqrt{g(\delta+1)} \\
&= \frac{1}{3}\sqrt{9n^2-12n\delta-42n+16{\delta^2+58\delta+61}} \\
&= \frac{1}{3}\sqrt{
9\left(n-\left\lfloor\frac{2}{3}\delta\right\rfloor-2\right)^2
-6n\left(2\delta -3\left\lfloor\frac{2}{3}\delta\right\rfloor+1\right)
-9\left\lfloor\frac{2}{3}\delta\right\rfloor\left(\left\lfloor\frac{2}{3}\delta\right\rfloor+4\right)
+16\delta^2+58\delta+25
} \\[4pt]
&\le \frac{1}{3}\sqrt{
9\left(n-\left\lfloor\frac{2}{3}\delta\right\rfloor-2\right)^2 -6n\left(2\delta -3\frac{(2\delta-1)}{3}+1\right)
-9\left(\frac{2\delta-2}{3}\right)\left(\frac{2\delta-2}{3}+4\right)
+16\delta^2+58\delta+25
}\\[4pt]
&= \frac{1}{3}\sqrt{
9\left(n-\left\lfloor\frac{2}{3}\delta\right\rfloor-2\right)^2
-12n+12\delta^2+42\delta+55
} \\[4pt]
&< n-\left\lfloor\frac{2}{3}\delta\right\rfloor-2 .
\end{aligned}
\]
which contradicts (3.2) \qed
\noindent Using an approach analogous to that of Remark 2.1, we obtain the following remark.\\

\noindent \textbf{Remark 3.1} The condition  $\rho(G) > \rho\left(K_\delta \vee \left( K_{\,n-\left\lfloor \frac{5}{3}\delta \right\rfloor -1}
\cup \left(\left\lfloor \frac{2}{3}\delta \right\rfloor +1\right)K_1 \right)\right)$ in Theorem 1.3 cannot be weakened to $\rho(G) \geq \rho\left(K_\delta \vee \left( K_{\,n-\left\lfloor \frac{5}{3}\delta \right\rfloor -1}
\cup \left(\left\lfloor \frac{2}{3}\delta \right\rfloor +1\right)K_1 \right)\right).$ To see this, consider the choice $S = V(K_\delta)$ in the graph $K_\delta \vee \left( K_{\,n-\left\lfloor \frac{5}{3}\delta \right\rfloor -1}
\cup \left(\left\lfloor \frac{2}{3}\delta \right\rfloor +1\right)K_1 \right)$. Then $|S| = \delta$, and we obtain, $i(G - S) = \left\lfloor \frac{2}{3}\delta \right\rfloor + 1 \geq \frac{2\delta}{3} + \frac{1}{3} = \frac{2|S|}{3} + \frac{1}{3} > \frac{2}{3}|S|$. Hence, by Theorem 1.1, the graph contains no $\{P_3, P_4, P_5\}$-factor. \qed
\noindent{\bf Conflict of interest.} The authors declare that they have no conflict of interest.\\

\noindent{\bf Data Availability} Data sharing is not applicable to this article as no datasets were generated or analyzed
during the current study.\\

\end{document}